\documentclass[a4paper]{amsart}
\usepackage[utf8]{inputenc}
\usepackage[english]{babel}
\usepackage [autostyle, english = american]{csquotes}
\MakeOuterQuote{"}

\usepackage{amssymb,mathrsfs,amsmath,graphicx,amsthm,microtype,tikz-cd,bbm,wasysym,enumitem,comment,cases,hyperref}
\usepackage{csquotes}

\usepackage[
backend=biber,
style=alphabetic,
sorting=nyt
]{biblatex}
 
\usetikzlibrary{babel}

\newtheorem{thm}{Theorem}[section]
\newtheorem{cor}[thm]{Corollary}
\newtheorem{lem}[thm]{Lemma}
\newtheorem*{lem*}{Lemma}
\newtheorem{defi}[thm]{Definition}

\newtheorem{prop}[thm]{Proposition}

\newtheorem{ej}[thm]{Example}
\newtheorem*{ej*}{Example}
\newtheorem*{thm*}{Theorem}
\newtheorem*{remark}{Remark}

\date{}
\title{Bounding the Local Dimension of the Convolution of measures}
\author{Kevin G. Hare}
\address{Department of Pure Mathematics \\
          University of Waterloo \\
          Waterloo, Canada}
\email{kghare@uwaterloo.ca}
\thanks{Research of K. G. Hare was supported, in part, by NSERC Grant 2019-03930}
\author{Joaquin G. Prandi}
\address{Department of Pure Mathematics \\
          University of Waterloo \\
          Waterloo, Canada}
\email{jgprandi@uwaterloo.ca}
\thanks{Research of J. G. Prandi was supported, in part, by NSERC Grant 2019-03930}

\def\R{\mathbb{R}}
\def\N{\mathbb{N}}

\def\Z{\mathbb{Z}}
\def\dim{\operatorname{dim}}
\def\supp{\operatorname{supp}}
\def\inte{\operatorname{int}}
\def\ldim{\underline{\operatorname{dim}}_{loc}}
\def\udim{\overline{\operatorname{dim}}_{loc}}

\newcounter{relctr} 
\everydisplay\expandafter{\the\everydisplay\setcounter{relctr}{0}} 

\AtBeginDocument{} 
\graphicspath{ {Pictures/}}

\begin{document}

\begin{abstract}
    We study the local dimension of the convolution of two measures. We give conditions for bounding the local dimension of the convolution on the basis of the local dimension of one of them. Moreover, we give a formula for the local dimension of some special points in the support of the convolution.
\end{abstract}
\maketitle
\section{Introduction}\label{intro}

The study of local dimensions has proven to be a rich field. The local dimension of a finite measure $\mu$ at a point $x$ in its support is defined as
\[\dim_{loc}\mu(x)=\lim_{r\to0}\frac{\log(\mu(B(x,r)))}{\log(r)},\]
if the limit exists. The local dimensions represent the amount of mass there is around a point in the support. The multifractal formalism (see \cite{falconer2})  provides a relation between the set of possible local dimensions of a measure and the Hausdorff dimension of the set of points in the support of a measure that achieves such dimension. 

Self-similar measures have particularly nice properties with respect to the multifractal formalism. If $\{S_1,\dots,S_n\}$ is an Iterated Function System (IFS) with contractions $r_1,\dots, r_n$ and probabilities $p_1,\dots,p_n$, then we associate the self-similar measure $\mu(A)=\sum_ip_i\mu\circ S_i^{-1}(A)$. Define $\beta(q)$ the solution to
\[\sum_ip_i^qr_i^{\beta(q)}=1\]
If the IFS satisfies the Open Set Condition (OSC), then the multifractal formalism concludes
\[\dim_H(\{x:\dim_{loc}\mu(x)=\alpha\})=\inf_q\{\alpha q+\beta(q)\}.\]
When a self-similar measure satisfies the OSC, the set of local dimensions is a closed interval, where all points of the interval are achieved. When the OSC fails, the situation becomes more complicated as first shown in \cite{HU20011}. In \cite{HU20011}, it was shown that the set of possible local dimensions of the Cantor measure convoluted with itself (enough times) has an isolated point. In other words, the set of possible local dimensions is a closed interval union a disjoint point. In \cite{Khare2011}, the results in \cite{HU20011} were extended to a larger family of Cantor-like measures. The convolution of the Cantor measure with itself satisfies the Finite Type Condition. In a sense, the finite type condition tells us that there are finitely many ways for the transformations to interact. For a precise definition, see \cite{hare/hare/rutar}.

In \cite{shmerkin2004modified}, a modified multifractal formalism was developed for a larger class of self-similar measures. In the process, techniques were introduced to calculate local dimensions. In \cite{HareHareMatthews, Hare_Hare_ShingNg_2018, HARE20181653}, a local dimension study was performed for self-similar measures that satisfy certain overlap conditions. In addition, in \cite{Shmerkin2016OnFI} and \cite{a34a4dbd-85aa-30eb-aa98-c3287493f960}, we see a strong relation between the convolutions and the $L^q$-spectrum. 

In \cite{Khare2011}, a general bound for convolutions was given for measures in $\R$, and we express this result in the case of groups. The result of \cite{Khare2011} has proven useful in studying the power of convolutions. In \cite{BruggemanCameron2013Maoc} it was shown that the set of local dimensions of convolution powers of a measure has an isolated point. If the support of a measure after convolution powers becomes an interval, then eventually the set of possible local dimensions of the convolution power will develop an isolated point. In Section \ref{Srealline}, we are able to reproduce this result using our new results.  We recover that the points in the support that could achieve this local dimension are on the boundary of the support. In \cite{BruggemanCameron2013Maoc} the interest was on the power of convolutions; here, we study the local dimension of the convolution of two, possibly different, measures. We are able to obtain the result that the local dimension of the convolutions of any two measures depends on the support of the measures and, in general, can be bounded by the local dimension of one of the measures.

Our main result is the following.
\begin{thm}\label{groups}
    Let $G$ be a group with a compatible metric $d$ that is translation-invariant. Let $\mu$ and $\nu$ be regular measures. Then if $\udim\mu(x)\leq\lambda$ for all $x\in \inte(\supp\mu)$ then $\udim(\mu*\nu)(z)\leq \lambda$ for all $z\in \inte(\supp\mu)\cdot\supp\nu=\{z=xy:x\in\inte(\supp\mu),y\in\supp\nu)\}$.
\end{thm}

The proof of this theorem is given in Section \ref{Sgroups}. We also prove some variations of the main theorem and see an application to the torus. In Section \ref{Srealline} we bring the study to the real line, where we work with measures with bounded support. Examples of why some conditions are important are given in this section. In Section \ref{Sspecailpoints} we will have results for special points on the real line. In Section \ref{COQuestions}, we present some open questions and conclusions. In the remainder of Section \ref{intro}, we recall the basic results and definitions.

Throughout this paper, we assume all measures to be Borel and finite.
\subsection{Basic results and definitions}

We first recall the definition of local dimension of a measure at a point $x$.
\begin{defi}\label{dlocal}
Given a measure $\mu$ on a metric space $X$, we define the \emph{upper local dimension of $\mu $} at a point $x\in \supp \mu$ as
\[\overline{\dim}_{loc}\mu(x)=\limsup_{r\to 0}\frac{\log(\mu(B(x,r)))}{\log(r)},\]
and the \emph{lower local dimension of $\mu$} at $x\in\supp \mu$ \[\underline{\dim}_{loc}\mu(x)=\liminf_{r\to 0}\frac{\log(\mu(B(x,r)))}{\log(r)}.\]
If these two values coincide, we call it the \emph{local dimension of $\mu$} at $x\in\supp \mu$, and write
\[\dim_{loc}\mu(x)=\lim_{r\to 0}\frac{\log(\mu(B(x,r)))}{\log(r)}.\]
\end{defi}

Recall the definition of convolution of two measures.
\begin{defi}\label{conv}
    Let $G$ be a topological group, and let $\mu$ and $\nu$ be regular measures defined on $G$. The \emph{convolution} $\mu*\nu$ is define as follow:
    \begin{align*}
        \mu*\nu(A)&=\int_G\int_G \chi_A(xy)d\mu(x)d\nu(y)\\
        &=\int_G\int_G \chi_{Ay^{-1}}(x)d\mu(x)d\nu(y)\\
        &=\int_G\mu(Ay^{-1})d\nu(y)\\
        &=\int_G\nu(x^{-1}A)d\mu(x).\\
    \end{align*}
\end{defi}
The following is a classic result and a proof can be found in \cite{bartle}.
\begin{lem}\label{lemabart}
    Let $\mu$ be a measure. If $(E_n)$ is a sequence of Borel sets such that $E_n\subset E_{n+1}$, then \[\mu\left(\bigcup_{n=1}^\infty E_n\right)=\lim_{n\to\infty}\mu(E_n).\]
\end{lem}

\section{Topological Groups}\label{Sgroups}
The proof of Theorem \ref{groups} is organized as follows: First, we define the subset of $\supp \mu$ and $\supp\nu$ that are of interest $M_z$ and $N_z$, respectively. Then we connect $N_z$ and $M_z$ with translations. Using this connection we are able to obtain a subset of $N_z$ where $\frac{\log(\mu(B(x,r)))}{\log(r)}<\lambda+\varepsilon$ for all $r<r_0$. We then use this to bound the convolution.
\begin{thm*}
    Let $G$ be a group with a compatible metric $d$ that is translation-invariant. Let $\mu$ and $\nu$ be regular measures. Then if $\udim\mu(x)\leq\lambda$ for all $x\in \inte(\supp\mu)$ then $\udim(\mu*\nu)(z)\leq \lambda$ for all $z\in \inte(\supp\mu)\cdot\supp\nu=\{z=xy:x\in\inte(\supp\mu),y\in\supp\nu)\}$
\end{thm*}
\begin{proof}
     Pick $z\in\inte(\supp\mu)\cdot\supp\nu$. Let
    \[M_z=\{x\in\inte(\supp\mu):\exists y\in\supp\nu:xy=z\},\]
and     \[N_z=\{y\in \supp\nu:\exists x\in\inte(\supp\mu):xy=z\}.\]
Let $\varphi_z(y)=zy^{-1}$ and $\varphi^z(x)=x^{-1}z$  then $\varphi^z(M_z)=N_z$ and $\varphi_z(N_z)=M_z$. Since $z\in\inte(\supp\mu)\cdot\supp\nu$, we have $N_z\neq \emptyset$. Note that $\varphi^z(\inte(\supp\mu))=(\inte(\supp\mu))^{-1}z$ is an open set and $\supp\nu\cap((\inte(\supp\mu))^{-1}z)=N_z$, hence $\nu(N_z)>0$.

Set $\delta:=\nu(N_z)/2$. We will show for all $\varepsilon>0$ there exists a set  $A_{\varepsilon}\subset N_z$ and $r_0$ such that
\begin{enumerate}
    \item $\nu(A_{\varepsilon})\geq\delta>0$.
    \item For all $x\in\varphi_z(A_{\varepsilon})$ and for all $r\leq r_0$ we have $\frac{\log(\mu(B(x,r)))}{\log(r)}<\lambda+\varepsilon$. 
\end{enumerate}
 Define 
 \[A_{\varepsilon,n}=\left\{x\in N_z:\frac{\log(\mu(B(\varphi_z(x),r)))}{\log(r)}\leq \lambda +\varepsilon \text{ for all }r<1/n\right\}.\]
 Then
\[N_z=\bigcup_{n=1}^\infty\left\{x\in N_z:\frac{\log(\mu(B(\varphi_z(x),r)))}{\log(r)}\leq \lambda +\varepsilon \text{ for all }r<1/n\right\}=\bigcup_{n=1}^\infty A_{\varepsilon,n}. \]
As $\nu(N_z)=2\delta$, we see that by Lemma \ref{lemabart} there exists some $N\in \N$ such that $\nu(A_{\varepsilon,N})\geq \delta$. Let $r_0(\varepsilon)=1/N$ and $A_\varepsilon=A_{\varepsilon,N}$. Then we see that $A_\varepsilon$ has properties $1$ and $2$ by the definition of $A_{\varepsilon,N}$.

Using this construction of $A_\varepsilon$ and $r_0$ we prove the result. Fix $\varepsilon>0$, and let $0<r<r_0$. Then,
\begin{align*}
    (\mu*\nu)(B(z,r))&=\iint \chi_{B(z,r)}(xy)d\mu(x)d\nu(y)\\
    &=\int\mu(B(z,r)y^{-1})d\nu(y)\\
    &\geq \int_{N_z}\mu(B(z,r)y^{-1})d\nu(y)\\
    &\geq\int_{A_\varepsilon}\mu(B(z,r)y^{-1})d\nu(y)\\
    &\geq\int_{A_\varepsilon}\inf_{t\in A_\varepsilon}\{ \mu(B(zt^{-1},r))\}d\nu(y)\\
    &=\inf_{t\in A_\varepsilon}\{ \mu(B(zt^{-1},r))\}\int_{A_\varepsilon}d\nu(y)\\
    &\geq\left(\inf_{x\in\varphi_z(A_\varepsilon)}\mu(B(x,r))\right)\delta,
\end{align*}
hence
\[\frac{\log((\mu*\nu)(B(z,r)))}{\log(r)}\leq\sup_{x\in\varphi_z(A_\varepsilon)}\frac{\log(\mu(B(x,r)))}{\log(r)}+\frac{\log(\delta)}{\log(r)}\leq \lambda+\varepsilon+\frac{\log(\delta)}{\log(r)}.\]
Letting $r\to 0$ we can conclude
\[\udim (\mu*\nu)(z)\leq \lambda+\varepsilon\Rightarrow\udim (\mu*\nu)(z)\leq \lambda.\]

\end{proof}
\begin{remark}
    We note that we could ask for translations to be Lipchitz, and with some simple modifications the proof works.
\end{remark}

\begin{thm}\label{generalb}
    Let $\mu$ and $\nu$ be measures defined in a metric group $G$, such that the metric is translation invariant. Then for $z\in \partial\supp\mu+\partial\supp\nu$,  \[\udim(\mu*\nu)(z)\leq \inf_{x_0y_0=z}\udim(\mu\times\nu)(x_0,y_0).\]
\end{thm}
\begin{proof}
Let $x_0y_0=z$. If $x\in B(x_0,r/2)$ and $y\in B(y_0,r/2)$, then $xy\in B(z,r)$, hence $\mu*\nu(B(z,r))\geq\mu(B(x_0,r/2))\nu(B(y_0,r/2))$. Then 
    \begin{align*}
        \udim(\mu*\nu)(z)&\leq \udim(\mu\times\nu)(x_0,y_0)\\
        \Rightarrow\udim(\mu*\nu)(z)&\leq\inf_{x_0y_0=z} \udim(\mu\times\nu)(x_0,y_0).\\        
    \end{align*}
\end{proof}

\begin{thm}
    Let $\mu$ and $\nu$ be measures in a metric group $G$, such that the metric is translation-invariant. If $\udim\mu(x)\leq \lambda$ for all $x\in S\subset\inte(\supp\mu)$ such that $\mu(G\setminus S)=0$, then $\udim\mu*\nu(z)\leq \lambda$ for all $z\in S\cdot\supp\nu$ such that $\nu(\varphi^z(S))>0$.
\end{thm}
\begin{proof}
    Let $z\in S\cdot\supp\nu$ be such that $\nu(G\setminus\varphi_z(S))>0$. Define
    \[M_z=\{x\in S:\exists y\in\supp\nu:xy=z\},\]
and     \[N_z=\{y\in \supp\nu:\exists x\in S:xy=z\}.\]
Let $\varphi_z(x)=zx^{-1}$ and $\varphi^z(x)=x^{-1}z$. Then $\varphi^z(M_z)=N_z$ and $\varphi_z(N_z)=M_z$. Since $z \in S\cdot\supp\nu$, we have $N_z\neq \emptyset$. Note that $\varphi^z(S)=(S)^{-1}z=N_z$ is such that $\nu(N_z)>0$.

The rest of the proof is the same as the proof of Theorem \ref{groups}.

\end{proof}
A measure $\mu$ is \emph{continuous} if $\mu(\{x\})=0$ for all $x\in G$.
\begin{cor}
    Let $\mu$ and $\nu$ be measures in a metric group $G$, such that the metric is translation-invariant. If $\mu$ and $\nu$ are continuous measures and $\udim\mu(x)\leq \lambda$ for all $x\in\supp\mu$ except maybe a countable set, then $\udim(\mu*\nu)(z)\leq \lambda$ for all $z\in \inte(\supp \mu)\cdot\supp\nu$.
\end{cor}

\begin{thm}\label{thmunique}
    Let $\mu$ and $\nu$ be measures in a metric group $G$ such that the metric is translation-invariant. If $z\in \supp\mu\cdot\supp\nu$ is such that there is a unique pair $x_0\in \supp\mu$ and $y_0\in \supp\mu$ such that $x_0y_0=z$ then 
    \[\udim(\mu*\nu)(z)=\udim(\mu\times\nu)(x_0,y_0).\]
\end{thm}
\begin{proof}

 Whenever $y\in B(y_0,r/2)$ and $ x\in B(x_0,r/2)$, we have $xy\in B(z,r)$. Hence \[(\mu*\nu)(B(z,r)\geq\mu(B(x_0,r/2))\nu(B(y_0,r/2)).\] In contrast, if $xy\in B(z,r)$ then $y\in B(y_0,r)$ and $x\in B(x_0,r)$. Thus \[(\mu*\nu)(B(z,r))\leq\mu(B(x_0,r))\nu(B(y_0,r)).\]
 From this two inequalities we obtain
    \begin{align*}
            \frac{\log(\mu(B(x_0,r))\nu(B(y_0,r)))}{\log(r)}& \leq\frac{\log((\mu*\nu)(B(z,r)))}{\log(r)}\\ &\leq \frac{\log(\mu(B(x_0,r/2))\nu(B(y_0,r/2)))}{\log(r)}.
        \end{align*}
        Taking $\limsup$ when $r\to0^+$ completes the proof.
\end{proof}
\begin{remark}
    Note that the same is true for $\ldim$ and $\dim_{loc}$.
\end{remark}
\begin{lem}\label{gkha}
    Let $G$ be a metric group such that translations are invariant under the metric. Let $\mu_1,\dots,\mu_n$ be measures on $G$, and set $\nu=\mu_1*\dots*\mu_n$. Let $z\in \supp\mu_1*\dots*\mu_n$  and let $x_i\in \supp\mu_i$ for $i=1,\dots,n$ are such that $x_1x_2\dots x_n=z$. Then 
    \[\udim\nu(z)\leq \sum_{i=1}^n\udim\mu_i(x_i).\]
\end{lem}
By $\mu^k$ we mean the convolution of $\mu$ with itself $k$ times.
\begin{thm}
    Let $\mu$ be a measure on $\mathbb{T}=\R/\Z$ such that $\udim\mu\leq\lambda$ for all $x\in \supp\mu$. If, in addition, if for some $N\in\N$ we have that $\supp\mu^N$ has non-empty interior. Then there exists $k_0\in \N$ such that for all $K\geq k_0$ we have \[\udim\mu^{N+k}(z)\leq \lambda N\ \forall z\in \mathbb{T}.\]
\end{thm}

\begin{thm}\label{lowerlocaldimensionbound}
    Let G be a metric group with a metric that is translation invariant.  Let $\mu$ and $\nu$ be regular measures. Suppose $\mu$ is such that $\ldim\mu(x)\geq\ \alpha$. Define
    \[B_{\varepsilon,n}^z=\left\{x\in \supp\nu:\frac{\log(\mu(B(\varphi_z(x),r)))}{\log(r)}\geq \alpha -\varepsilon \text{ for all }r<1/n\right\}.\]
    then 
    \[\ldim\mu*\nu(z)\geq \alpha\ \forall 
    z\in\supp\nu\cdot\supp\nu.\]
      such that for all $\varepsilon>0$ there is $n_0\in \N$ such that
    \begin{equation}\label{eqforuniforlowerbound}
        \nu(G\setminus B_{\varepsilon,n_0}^z)=0,
    \end{equation}
    
\end{thm}
\begin{proof}
    Let $z\in\supp\nu\cdot\supp\nu$ be such that for all $\varepsilon>0$, there exists $n_0\in \N$ such that the equation \eqref{eqforuniforlowerbound} is true. Let $0<r<n_0$

    \begin{align*}
    (\mu*\nu)(B(z,r))&=\iint \chi_{B(z,r)}(xy)d\mu(x)d\nu(y)\\
    &=\int\mu(B(z,r)y^{-1})d\nu(y)\\
    &=\int_{B_{\varepsilon,n_0}^z}\mu(B(z,r)y^{-1})d\nu(y)\\
    &\leq\int_{B_{\varepsilon,n_0}^z}\sup_{t\in B_{\varepsilon,n_0}^z}\{ \mu(B(zt^{-1},r))\}d\nu(y)\\
    &=\sup_{t\in B_{\varepsilon,n_0}^z}\{ \mu(B(zt^{-1},r))\}\int_{B_{\varepsilon,n_0}^z}d\nu(y)\\    &=\left(\inf_{x\in\varphi_z(B_{\varepsilon,n_0}^z)}\mu(B(x,r))\right)\nu(B_{\varepsilon,n_0}^z),
\end{align*}

\[\frac{\log((\mu*\nu)(B(z,r)))}{\log(r)}\geq\inf_{x\in\varphi_z(B_{\varepsilon,n_0}^z}\frac{\log(\mu(B(x,r)))}{\log(r)}+\frac{\log(\nu(B_{\varepsilon,n_0}^z))}{\log(r)}\geq \alpha+\varepsilon+\frac{\log(\nu(B_{\varepsilon,n_0}^z)}{\log(r)}.\]
Letting $r\to 0$ we can conclude
\[\ldim (\mu*\nu)(z)\geq \alpha-\varepsilon\Rightarrow\udim (\mu*\nu)(z)\geq \alpha.\]
\end{proof}
\section{The Real Line}\label{Srealline}
On the real line, we will concentrate on measures with bounded support.
\begin{defi}
    Let $\mu$ be a finite measure with $\supp\mu$. We say $(c,d)\subset\operatorname{hull}\supp\mu$ is a \emph{gap} in $\supp \mu$ if $c,d\in \supp\mu$ and $(c,d)\cap \supp\mu=\emptyset$. 
\end{defi}
\begin{thm}\label{Theorem1}
Let $\mu$ and $\nu$ be measures, with $\supp \mu=[0,1]$ and $0,D\in \supp \nu\subset [0,D]$. Furthermore, assume that the largest gap in $\supp \nu$ has a diameter less than $1$. 
If $\udim \mu(x)\leq \lambda<\infty$ for all $x\in (0,1)$, then $\udim (\mu*\nu)(z)\leq \lambda$ for all $z\in (0,D+1)$. 
\end{thm}
\begin{remark}
    We note that by rescaling and translation we can change $\supp \mu$ from any interval $[a,b]$ to $[0,1]$, without changing the result. 
\end{remark}
\begin{proof}
    Note that $(0,D+1)=\inte(\supp\mu)+\supp\nu$, so we can apply Theorem \ref{groups}.
\end{proof}

The next example illustrates that we can weaken the condition on $\udim\mu$. 
\begin{ej}\label{example1}
    Consider the IFS given by $\{S_i\}_{i=0}^2$ with $S_i=x/3+i/3$. The attractor of this IFS is $[0,1]$, we consider the measure given by \[\mu(A)=\frac{2}{5}\mu\circ S_0^{-1}(A)+\frac{1}{5}\mu\circ S_1^{-1}(A)+\frac{2}{5}\mu\circ S_2^{-1}(A).\]
    Using the techniques of \cite{falconer2}, we see that $\udim\mu(x)\leq \log(5)/\log(3)$. But it is also true that $\mu$ is exact dimensional, i.e., for almost all points $x$ with respect to $\mu$ and the Lebesgue measure has $\udim \mu(x)=\log(4/125)/\log(1/27)$.

    Now, consider $\mu*\mu$. According to Theorem \ref{Theorem1} we have $\udim\mu*\mu(z) \leq \log(5)/\log(3)$ for $z\in (0,2)$. Note that $\udim \mu*\mu(1)\leq \log(4/125)/\log(1/27)<\log(5)/\log(3)$. 

     We will give the details of the example after the next theorem.
\end{ej}
\begin{remark}
    We note that the IFS defined on the above example satisfies the OSC.
\end{remark}

\begin{thm}\label{notallpoints}
     Let $\mu$ and $\nu$ be measures, with $\supp \mu=[0,1]$ and $0,D\in\supp \nu\subset[0,D]$. Furthermore, assume that the largest gap in $\supp\nu$ has diameter less than $1$. If $\udim\mu(x)\leq \lambda$ for all $x\in S$ with $\mu([0,1]\setminus S)=0$. Then $\udim(\mu*\nu)(z)\leq \lambda$ for all $z\in (0,D+1)$  such that $\nu(\varphi_z(S))>0$.
\end{thm}
\begin{ej}[Details of Example \ref{example1}]

Let $B$ be the set of points such that $\udim\mu(x)\neq\log(4/125)/\log(1/27)$, then $\mu(B)=0$.

Note that $\varphi_1(B)=1-B=B$. This is due to $\mu$ being a symmetric measure with respect to $1/2$ and $\varphi_1$ being equivalent to reflecting on $1/2$, therefore $\mu(\varphi_1(B))=0$. Then, as before, $N_z=\{y\in \supp\,u:\exists x\in(0,1)\setminus B:x+y=z\}$ has a positive measure $\mu$. 
\end{ej}
The next example shows that the conditions of Theorem \ref{notallpoints} are not always met. 
\begin{ej}
    Consider the IFS given by $F_0(x)=\frac{x}{2}$ and $F_1(x)=\frac{x+1}{2}$. Let $p\in(0,1/2]$, we define
    \begin{equation*}
        \mu_p(A)=p\mu_p\circ F_0^{-1}(A)+(1-p)\mu_p\circ F_1^{-1}(A).
    \end{equation*}
    Let $n(x|_k)$ be the number of occurrences of the digit $0$ in the first $k$ digits of the binary expression of $x$. Let \[K_p=\left\{x\in[0,1]:\lim_{k\to\infty}\frac{n(x|_k)}{k}=p\right\}.\]
    Then, as shown in \cite{falconer2}, $\mu_p(K_p)=1$ and for all $x\in K_p$ we have \[\udim\mu_p(x)=\frac{-(p\log(p)+(1-p)\log(1-p))}{\log(2)}=s(p).\] Consider $\mu_{1/4}$ and $\mu_{1/3}$. Then \[B=\left\{x\in[0,1]:\udim\mu_{1/3}\neq s(1/3)\right\}=\left\{x\in[0,1]:\lim_{k\to\infty}\frac{n(x|_k)}{k}\neq 1/3\right\}.\]
    Moreover, $\mu_{1/3}(B)=0$. Note that $K_{1/4}\subset \varphi_1(B)$, we have $\mu_{1/4}(B)=1$, and hence $\mu_{1/4}(\{y\in \supp\nu:\exists x\in(0,1)\setminus B:x+y=1\})=0$. Hence, the points in $K_{1/3}^c$ are important for $\udim\mu_{1/3}*\mu_{1/4}$.
\end{ej}
\begin{thm}\label{theo2}
    Let $\mu$ and $\nu$ be measures, with $\supp \mu=[0,1]$ and $0,D\in\supp \nu\subset[0,D]$. Furthermore, assume that the largest gap in $\supp\nu$ has diameter less than $1$. If $\mu$ and $\nu$ are continuous measures, that is, $\mu(\{x\})=\nu(\{x\})=0$ for all $x$. If $\udim\mu(x)\leq \lambda$ for all $x\in\supp\mu$ except maybe a countable set. Then $\udim(\mu*\nu)(z)\leq \lambda$ for all $z\in (0,D+1)$.
\end{thm}
In the next example, we show that the bound is sharp. 
\begin{ej}
    Let $\mu$ be the Lebesgue measure on $[0,1]$ and $\nu$ be a measure such that $0,D\in\supp\nu\subset [0,D]$, such that the largest gap has diameter less than 1. Then   
    \begin{equation}
        \dim_{loc}\mu*\nu(z)=1\ \forall z\in (0,D+1)
    \end{equation}
    Note that from Theorem \ref{Theorem1} we already know that $\udim\mu*\nu(z)\leq 1$ for all $z\in (0,D+1)$. For the lower bound, we proceed in similar ways as in our theorems. Let $z\in \supp\mu*\nu$ and $\varepsilon>0$ and $z\neq 0,D+1$ 
    Let \[M_z=\{x\in (0,1):\exists y\in\supp\nu:x+y=z\}\]
and     \[N_z=\{y\in \supp\nu:\exists x\in(0,1):x+y=z\}\]

Let $\varphi_z(x):x\mapsto z-x$. Then we have $\nu(N_z)>0$. Note that for all $x\in N_z$ and all $r>0$ we have $\mu(B(\varphi_z(x),r))\leq 2r$ Then we have
\begin{align*}
    (\mu*\nu)(B(z,r))&=\iint \chi_{B(z,r)}(x+y)d\mu(x)d\nu(y)\\
    &=\int_{M_z\cup \{0,1\}}\nu(B(z,r)-x)d\mu(x)\\
    &=\int_{M_z}\nu(B(z,r)-x)d\mu(x)\\
    &=\int_{N_z}\mu(B(z,r)-y)d\nu(y)\\
    &\leq\int_{N_z}2r d\nu(y)\\
    &\leq 2r \int_{N_z} d\nu(y).
\end{align*}
Applying logarithms to both sides gives us
\[\frac{\log((\mu*\nu)(B(z,r)))}{\log(r)}\geq\frac{\log(2 r\nu(N_z))}{\log(r)}.\]
So, we may conclude
\[\dim_{loc}\mu*\nu(z)=1\ \forall z\in (0,D+1)\]
We also note that by Theorem \ref{thmunique}, we can make a complete study the local dimension $\mu*\nu$, since it gives the local dimension at $0$ and $D+1$.
\end{ej}

\section{Local Dimension at Special Points}\label{Sspecailpoints}
\begin{thm}\label{boundar}
    Let $\mu$ and $\nu$ be measures with  $0\in\supp\mu,\supp \nu\subset [0,1]$, then
    \[\overline{\dim}_{loc}(\mu*\nu)(0)=\overline{\dim}_{loc}((\mu\times\nu))(0,0).\]
\end{thm}
\begin{remark}$\ $
\begin{itemize}
    \item The above result is also true if we change $\udim$ to $\ldim$.

    \item It is clear from the definitions that:
    \[\overline{\dim}_{loc}(\mu\times\nu)(0,0)\leq \overline{\dim}_{loc}\mu(0)+\overline{\dim}_{loc}\nu(0),\]
    and
    \[ \underline{\dim}_{loc}(\mu\times\nu)(0,0)\geq \underline{\dim}_{loc}\mu(0)+\underline{\dim}_{loc}\nu(0).\]
    \item If $\mu=\nu$ then 
    \[\overline{\dim}_{loc}(\mu\times\nu)(0,0)=2\udim\mu(0),\]
    and
    \[\ldim\mu\time\nu(0,0)=2\ldim\mu(0).\]
    \item If $0,D\in \supp\nu\subset [0,D]$ we can state something similar for $\udim(\mu*\nu)(D+1)$.
\end{itemize}
\end{remark}

By $(N)\supp\mu$ we mean the $N$-fold sum of $\supp \mu$.
\begin{thm}\label{A}
     Let $\mu$ be such that $0,1\in\supp\mu\subset [0,1]$, with $\udim\mu(x)\leq \lambda<\infty$ and $\udim\mu(0)>0$. In addition, assume $(N)\supp\mu=[0,N]$, then there is $K\in \N$ such that for all $k\geq K$, $\udim\mu^k(0)$ is an isolated point in the set of upper local dimensions of $\mu^k$.
\end{thm}
The proof of this is essentially the same as the one found in \cite{BruggemanCameron2013Maoc}, the main difference being that we do not ask for $\mu$ to be a continuous measure.
\begin{cor}
    Let $\mu$ and $\nu$ be measures, with $\supp \mu=[0,1]$ and $0,D\in \supp \nu\subset [0,D]$. Furthermore, assume that the largest gap $\supp \nu$ has a diameter less than $1$.  If $\udim \mu(x)\leq \lambda<\infty$ for all $x\in (0,C)$, with $C<1$, then  $\udim (\mu*\nu)(z)\leq \lambda$  for all $z\in (0,C)$.
\end{cor}
\begin{proof}
  The proof is essentially the same as Theorem \ref{Theorem1}, note that if $z\in (0,C)$ then, as before, define $N_z$ and note that $\nu(N_z)>0$ such that $0\notin \varphi(N_z)$. In this case, $A_\varepsilon\subset \varphi(N_z)\subset [0,C)$.

\end{proof}
\begin{cor}
    Let $\mu$ and $\nu$ be measures, with $\supp \mu=[0,1]$ and $0,D\in \supp \nu\subset [0,D]$. Furthermore, assume that the largest gap $\supp \nu$ has a diameter less than $1$. If $\udim \mu(x)\leq \lambda<\infty$ for all $x\in (1-C,1)$, then  $\udim (\mu*\nu)(z)\leq \lambda$  for all $z\in (1+D-C,1+D)$. 
\end{cor}
\begin{proof}
Consider the measures $\mu'(X)=\mu(1-X)$ and $\nu'(X)=\nu(D-X)$ and apply the previous corollary.
\end{proof}

Next, we show an example of why the condition over the gaps is important. This example shows that when the boundaries of $\supp\mu$ and $\supp\nu$ are the only parts that interact, the local dimension of $\mu*\nu$ is determined by the local dimension of the boundaries. 
\begin{ej}
    Consider the measure $\nu=\delta_0+\delta_2$ where $\delta_a(A)$ is $1$ if $a\in A$ and $0$ otherwise. Let $\mu$ be the convolution of the Lebesgue measure restricted to $[0,1]$ with itself. Note that \[\mu(A)=\int_A f(x)dx \text{ where }f(x)= \left\{ \begin{array}{cc}     x   &\text{ if } x\in[0,1] \\       2-x   &\text{ if }x\in(1,2]\\       0& \text{otherwise.}       \end{array}\right.\] then for all $x\in(0,2)$ we have $\dim_{loc}\mu(x)=1$. 

The gap in the support of $\nu$ is exactly the length of $\supp \mu$, so
\[(\mu*\nu)(A)=\int_A g(x)dx\text{ where }g(x)= \left\{ \begin{array}{cc}
    x   &\text{ if } x\in[0,1] \\
      2-x   &\text{ if }x\in(1,2]\\
      x-2   &\text{ if } x\in(2,3] \\
      4-x   &\text{ if }x\in(3,4]\\
      0& \text{otherwise.}
      \end{array}\right.\]
Hence the local dimension of $(\mu*\nu)$ at $x=2$ is $2$. 

\end{ej}
\begin{thm}\label{th}
 Let $\mu$ and $\nu$ be measures, with $\supp \mu=[0,1]$ and $0,D\in \supp \nu\subset [0,D]$. Suppose $[b,c]\subset [0,D]$ is a gap in $\supp\nu$. Then,
 \begin{enumerate}[label=(\roman*)]
     \item if $c-b>1$, then \[\ldim(\mu*\nu)(1+b)=\ldim(\mu\times\nu)(1,b),\]
     \item\label{ii} if $c-b=1$, then \[\ldim(\mu*\nu)(1+b)=\min\left\{\ldim(\mu\times\nu)(1,b),\ldim(\mu\times\nu)(0,c) \right\}.\]
 \end{enumerate}
\end{thm}
\begin{proof}
Part $(i)$ is Corollary \ref{thmunique}. Part $(ii)$ will be proven as part of Proposition \ref{prop12}
\end{proof}
\begin{cor}
     Let $\mu$ and $\nu$ be measures, with $\supp \mu=[0,1]$ and $0,D\in \supp \nu\subset [0,D]$. Suppose $[b,c]\subset [0,D]$ is a gap in $\supp\nu$. If $c-b>1$, then \[\udim(\mu*\nu)(1+b)=\udim(\mu\times\nu)(1,b)\]
\end{cor}
\begin{cor}
    Let $\mu$ and $\nu$ be measures, with $\supp \mu=[0,1]$ and $0,D\in \supp \nu\subset [0,D]$. Suppose $[b,c]\subset [0,D]$ is a gap in $\supp\nu$. If $c-b>1$, then \[\ldim(\mu*\nu)(1+b)=\ldim(\mu\times\nu)(0,c),\]
    and
    \[\udim(\mu*\nu)(c)=\udim(\mu\times\nu)(0,c).\]
\end{cor}
\begin{prop}\label{prop12}
     Let $\mu$ and $\nu$ be measures, with $\supp \mu=[0,1]$ and $0,D\in \supp \nu\subset [0,D]$. Suppose $[b,c]\subset [0,D]$ is a gap in $\supp\nu$. Then,
 if $c-b=1$, then 
 \begin{align*}
     \udim(\mu*\nu)(c)&=\limsup_{r\to 0}\frac{\log((\mu\times\nu)(B((0,c),r))+(\mu\times\nu)(B((1,b),r)))}{\log(r)}\\&\leq\min\{\udim(\mu\times\nu)(0,c),\udim(\mu\times\nu)((1,b))\}.
 \end{align*}
\end{prop}
\begin{proof}
     If $x+y\in B(1+b,r)$ then $y\in B(b,r)$ and $x\in B(1,r/2)$ or $y\in B(c,r)$ and $x\in B(0,r)$, thus \[(\mu*\nu)(B(c,r)\leq\mu(B(1,r))\nu(B(b,r))\text{ or }(\mu*\nu)(B(c,r))\leq\mu(B(0,r))\nu(B(c,r)).\] Then,
    \[(\mu*\nu)(B(c,r))\leq\mu(B(1,r))\nu(B(b,r))+\mu(B(0,r))\nu(B(c,r)).\]
    On the other hand, $x+y\in B(c,r)$ whenever $y\in B(b,r/2), x\in B(1,r/2)$. So, \[(\mu*\nu)(B(c,r))\geq\mu(B(1,r/2)\nu(B(b,r/2)).\] Similarly \[(\mu*\nu)(B(c,r)\geq\mu(B(0,r/2)\nu(B(c,r/2)).\] Hence
    \[\mu(B(1,r/2))\nu(B(b,r/2))+\mu(B(0,r/2))\nu(B(c,r/2))\leq 2(\mu*\nu)(B(c,r))\]
    Which proves the equality. 
   
    Let $\varepsilon>0$, $\udim (\mu\times\nu)(0,c)=M$ and $\udim(\mu\times\nu)(1,b)=L$. Then $\overline{\dim}_{loc}(\mu\times\nu)(0,c)=M$, given $\varepsilon>0$ exist $r_1$ such that for all $r_0\geq r>0$
 \[\frac{\log((\mu\times\nu)(B((0,c),r))}{\log(r)}\leq M+\varepsilon,\]
Then
\[(\mu\times\nu)(B((0,c),r))\geq r^{M+\varepsilon}.\]
Similarly, there exists $r_2$ such that for all $r_1\geq r>0$ we have 
\[(\mu\times\nu)(B((1,b),r))\geq r^{L+\varepsilon}.\]

Suppose that $\min\{M,L\}=M$, let $0<r\leq \min\{r_1,r_2\}$. Then
\begin{align*}
   \limsup_{r\to 0}&\frac{\log((\mu\times\nu)(B((0,c),r))+(\mu\times\nu)(B((1,b),r)))}{\log(r)}\\
   &\leq\limsup_{r\to 0}\frac{\log(r^{M+\varepsilon}+r^{L+\varepsilon})}{\log(r)}\\
   &\leq\limsup_{r\to 0}\frac{\log(r^{M+\varepsilon}(1+r^{L-M})}{\log(r)}\\
   &\leq\limsup_{r\to 0}\frac{\log(r^{M+\varepsilon})+\log(1+r^{L-M})}{\log(r)}
   =M+\varepsilon.
\end{align*}
and we conclude
\begin{align*}
     \udim(\mu*\nu)(c)&=\limsup_{r\to 0}\frac{\log((\mu\times\nu)(B((0,c),r))+(\mu\times\nu)(B((1,b),r)))}{\log(r)}\\&\leq\min\{\udim(\mu\times\nu)(0,c),\udim(\mu\times\nu)((1,b))\}.
 \end{align*}
\end{proof}

\begin{remark}
We note that this proposition is the same as the one in \cite{hare/prandi}, where the local dimension of the additions was studied. More can be said about when this inequality is an equality; we refer the reader to that paper.
\end{remark}
\section{Conclusions and Open Questions}\label{COQuestions}

Intuitively, the local dimension at a point $x$ is how much mass there is around said point. The bigger the local dimension, the less mass there is. The convolution of two measures redistributes the masses of both measures into a new measure. With this in mind, it is not surprising that we can bound the upper local dimension of the convolution on the basis of one of the measures. In general, we observe that the mass is distributed in a relatively uniform way for the points in the interior of $\supp\mu*\nu$, so that all the interior points obtain a minimum amount of mass. On the other hand, how much mass can a point receive is a more complicated question. Although we were able to obtain Theorem \ref{lowerlocaldimensionbound}, the conditions we ask are strong which make our proof possible. When working with particular measures, showing that a given measure has the requirements to apply Theorem \ref{lowerlocaldimensionbound} can be a challenge. A natural next question to ask is:

\emph{
    Is there a weaker or different condition for Theorem \ref{lowerlocaldimensionbound}. If we know $\ldim\mu(x)\geq\lambda$, under what conditions is it true that $\ldim \mu*\nu(z)\geq \lambda$.  }

We worked with groups and the real line. Our results from groups translate to $\R^n$ (or vector spaces in general). On the other hand, the results from the real line strongly use the underlying structure of $\R$. This causes some problem when we want to work on $\R^n$. We could define the gaps in $\R^n$ as follows:
\begin{defi}
    Let $\nu$ be a measure on $\R^n$ of bounded support. We say $A\subset\operatorname{hull}(\supp\nu)$ is a gap in $\supp\nu$ if
    \begin{itemize}
        \item $A$ is connected.
        \item $A$ is open.
        \item $A\cap (\supp\nu)=\emptyset$
    \end{itemize}
\end{defi}
In contrast to working in $\R$, the gaps in $\R^n$ can have any shape. In addition, to be able to cover any gap, we need the support of $\mu$ to be a ball. Then we could recover the results of Section \ref{Srealline}, but with the $z\in B(0,1)+\supp\nu$. 

The results of Section \ref{Sspecailpoints}, would still be highly dependent on the shape of the gaps. When working on $\R^n$, there is the question of what happens to points $z\in\supp\mu*\nu$, such that $z\in\partial\supp\mu+\partial\supp\nu\setminus(\inte(\supp \mu)+\supp \nu)$. Theorem \ref{boundar} answers the question of what happens in $\R$, if the supports of both $\mu$ and $\nu$ are convex. In $\R^n$, even if the supports are convex, things get complicated. As an example, let $\supp\mu=\supp\nu=[0,1]\times[0,1]$. Then $\supp(\mu*\nu)=[0,2]\times[0,2]$. If $z\in (0,2)\times\{0\}$, it is not clear what a good bound for $\udim(\mu*\nu)(z)$ would be. Theorem \ref{generalb} gives an upper bound; but in practice, it is not a good upper bound, since it requires one to calculate a new measure and then the local dimension of that new measure. 

From the work we present, we deduce that the "number of points" behind $z\in\supp\mu*\nu$ is key. If $z\in \inte(\supp\mu)+\supp\nu$ is a way to say that there are enough points such that $x+y=z$, since we have an open set of points in $\supp\mu$ from where to choose the points that have the properties we want, as we do in theorems earlier in this paper. In other cases where there are finitely many $x\in\supp\mu$ and $y\in \supp\nu$ such as $x+y=z$, then we can work around and do something, as in Section \ref{Sspecailpoints}. In other words, in the finite case, there are enough constraints that we have something to work with. It is an open question as to what happens in the middle ground between finiteness and contained in an open set. What happens if there are countably many $x\in\supp\mu $ and $y\in\supp\nu$ such that $x+y=z$? 

Although there are some unanswered questions and some difficulties with our techniques. We believe that the techniques we develop here could be modified or improved for more results in new settings. 
\nocite{Bruggeman_2014}
\printbibliography[heading=bibintoc,title={References}]
\end{document}